\documentclass[12pt]{article}
\usepackage{amsmath,amsfonts,amssymb,amscd,verbatim,comment, amsthm}
\usepackage{fullpage}
\usepackage{color}
\usepackage{array}
\usepackage{enumerate}
\usepackage{mathtools}
\allowdisplaybreaks
\numberwithin{equation}{section}
\theoremstyle{plain}
 
\newtheorem{theorem}[equation]{Theorem} 
\newtheorem{conjecture}[equation]{Conjecture}
\newtheorem{corollary}[equation]{Corollary} 
\newtheorem{lemma}[equation]{Lemma}

\theoremstyle{definition}
\newtheorem{definition}[equation]{Definition}

\title{On the spread of outerplanar graphs}
\author{Daniel Gotshall\thanks{Department of Mathematics \& Statistics, Villanova University, email addresses \texttt{\{dgotshall, mobrie71, michael.tait\}@villanova.edu}. The third author is partially supported by National Science Foundation grant DMS-2011553.} \and Megan O'Brien\footnotemark[1] \and Michael Tait\footnotemark[1]}

\begin{document}

\maketitle
\begin{abstract}
    The spread of a graph is the difference between the largest and most negative eigenvalue of its adjacency matrix. We show that for sufficiently large $n$, the $n$-vertex outerplanar graph with maximum spread is a vertex joined to a linear forest with $\Omega(n)$ edges. We conjecture that the extremal graph is a vertex joined to a path on $n-1$ vertices.
\end{abstract}

\section{Introduction}
The {\em spread} of a square matrix $M$ is defined to be 
\[
S(M): = \max_{i,j}|\lambda_i - \lambda_j|,
\]
where the maximum is taken over all pairs of eigenvalues of $M$. That is, $S(M)$ is the diameter of the spectrum of $M$. The spread of general matrices has been studied in several papers (e.g. \cite{B, Deutsch, JKW, Mirsky, NylenTam, Thompson, WZL}). In this paper, given a graph $G$ we will study the spread of the adjacency matrix of $G$, and we will call this quantity the {\em spread of $G$} and denote it by $S(G)$. Since the adjacency matrix of an $n$-vertex graph is real and symmetric, it has a full set of real eigenvalues which we may order as $\lambda_1 \geq \cdots \geq \lambda_n$. In this case, the spread of $G$ is given simply by $\lambda_1 - \lambda_n$.

The study of the spread of graphs was introduced in a systematic way by Gregory, Hershkowitz, and Kirkland in \cite{GHK}. Since then, the spread of graphs has been studied extensively. A problem in this area with an extremal flavor is to maximize or minimize the spread over a fixed family of graphs. This problem has been considered for trees \cite{AP}, graphs with few cycles \cite{FWG, PBA, WS}, the family of all $n$-vertex graphs \cite{variable, BRTU, alex, stanic, stevanovic, john}, bipartite graphs \cite{BRTU}, graphs with a given matching number \cite{LZZ} or girth \cite{WZS} or size \cite{LM}.

We also note that spreads of other matrices associated with a graph have been considered extensively (e.g. \cite{ADLR, FF, FWL, GZ, LMG, LL, OLNK,  YRY, YL,YZLWS}), but in this paper we will focus on the adjacency matrix. This paper examines the question of maximizing the spread of an $n$-vertex outerplanar graph. A graph is {\em outerplanar} if it can be drawn in the plane with no crossings and such that all vertices are incident with the unbounded face. Similarly to Wagner's theorem characterizing planar graphs, a graph is outerplanar if and only if it does not contain either $K_{2,3}$ or $K_4$ as a minor. Maximizing the spread of this family of graphs is motivated by the extensive history on maximizing eigenvalues of planar or outerplanar graphs, for example \cite{BR, CV, CR, DM, EZ, HS, LN, SW, Rowlinson, yuan1, yuan2}.

Our main theorem comes close to determining the outerplanar graph of maximum spread. Let $P_k$ denote the path on $k$ vertices and $G\vee H$ the join of $G$ and $H$. A linear forest is a disjoint union of paths.

\begin{theorem}\label{thm: main}
For $n$ sufficiently large, any graph which maximizes spread over the family of outerplanar graphs on $n$ vertices is of the form $K_1 \vee F$ where $F$ is a linear forest with $\Omega(n)$ edges. 
\end{theorem}

We leave it as an open problem to determine whether or not $F$ should be a path on $n-1$ vertices, and we conjecture that this is the case.

\begin{conjecture}
For $n$ sufficiently large, the unique $n$-vertex outerplanar graph of maximum spread is $K_1 \vee P_{n-1}$.
\end{conjecture}

\section{Preliminaries}
Let $G$ be an outerplanar graph of maximum spread and let $A$ be its adjacency. We will frequently assume that $n$ is sufficiently large. We will use the characterization that a graph is outerplanar if and only it does not contain $K_{2,3}$ or $K_4$ as a minor. In particular, $G$ does not contain $K_{2,3}$ as a subgraph.  Given a vertex $v\in V(G)$, the neighborhood of $v$ will be denoted by $N(v)$ and its degree by $d_v$. If $f,g:\mathbb{N} \to \mathbb{R}$ we will use $f = \mathcal{O}(g)$ to mean that there exists a constant $c$ such that $f(n) \leq cg(n)$ for $n$ sufficiently large. $f = \Omega(g)$ means that $g = \mathcal{O}(f)$ and $f = \Theta(g)$ means that $f = \mathcal{O}(g)$ and $g = \mathcal{O}(f)$. We will occasionally have sequences of inequalities where we will abuse notation and mix inequality symbols with $\mathcal{O}(\cdot)$ and $\Theta(\cdot)$. \\

Let the eigenvalues of $A$ be represented as $\lambda_1\geq \lambda_2\geq \cdots \geq \lambda_n$. For any disconnected graph, adding an edge between the connected components will not decrease $\lambda_1$ and will also not decrease $-\lambda_n$. Therefore, without loss of generality, we may assume that $G$ is connected. By the Perron-Frobenius theorem we may assume that the eigenvector $\mathbf{x}$ corresponding to $\lambda_1$ has $\mathbf{x}_u > 0$ for all $u$.

Furthermore, we will normalize $\mathbf{x}$ so that it has maximum entry equal to $1$, and let $\mathbf{x}_w=1$ where $w$ is a vertex attaining the maximum entry in $\mathbf{x}$. Note there may be more than one such vertex, in which case we can arbitrarily choose and fix $w$ among all such vertices. The other eigenvector of interest to us corresponds to $\lambda_n$, call it $\mathbf{z}$.  We will also normalize $\mathbf{z}$ so that its largest entry in absolute value has absolute value $1$ and let $w'$ correspond to a vertex with maxumum absolute value in $\mathbf{z}$ (so $\mathbf{x}_{w'}$ equals $1$ or $-1$).

We will implement the following known equalities for the largest and smallest eigenvalues:
\begin{align}
    \lambda_1=\max\limits_{\mathbf{x}'\neq 0}&\frac{(\mathbf{x}')^tA\mathbf{x}'}{(\mathbf{x}')^t\mathbf{x}'}=\frac{\mathbf{x}^tA\mathbf{x}}{\mathbf{x}^t\mathbf{x}}\label{Rayleigh max}\\
    \lambda_n=\min\limits_{\mathbf{z}'\neq 0}&\frac{(\mathbf{z}')^tA\mathbf{z}'}{(\mathbf{z}')^t\mathbf{z}'}=\frac{\mathbf{z}^tA\mathbf{z}}{\mathbf{z}^t\mathbf{z}}\label{Rayleigh min}
\end{align}\\

An important result from equations \ref{Rayleigh max} and \ref{Rayleigh min} and the Perron-Frobenius Theorem is that for any strict subgraph $H$ of $G$, we have $\lambda_1(A(G))>\lambda_1(A(H))$. 
Finally, we will use the following theorem from \cite{TT}

\begin{theorem}\label{taittobin}
For $n$ large enough, $K_1 \vee P_{n-1}$ has maximum spectral radius over all $n$-vertex outerplanar graphs.
\end{theorem}

\section{Vertex of Maximum Degree}
The main goal of this section is to prove that $G$ has a vertex of degree $n-1$. This is stated in the following theorem. 
\begin{theorem}\label{vertex degree n-1}
For $n$ large enough, we have $d_w=n-1$.
\end{theorem}
As a first step, we will get preliminary upper and lower bounds on the largest and smallest eigenvalues of $G$. First we obtain an upper bound on the spectral radius.

 \begin{lemma}\label{lambda_1 bounds}
 $\lambda_1 \leq \sqrt{n}+1$. 
\end{lemma}
\begin{proof}
We define the graph $G_1$ to be the graph $K_1 \vee P_{n-1}$. By Theorem \ref{taittobin}, we know that any outerplanar graph on sufficiently many vertices cannot have a  spectral radius larger than that of $G_1$. Now define $G_2$ as $G_1$ with another edge joining the endpoints of the path, so $G_2 = K_1 \vee C_{n-1}$. Clearly $G_1$ is a subgraph of $G_2$. Putting all this together gives us 
\begin{align*}
    \lambda_1(G)\leq\lambda_1(G_1)<\lambda_1(G_2)= \sqrt{n}+1,
\end{align*}
where the last equality can be calculated using an equitable partition with two parts (the dominating vertex and the cycle).
\end{proof}
Next we bound $|\lambda_n|$.
\begin{lemma}\label{lambda_n crude bound}
For $n$ sufficiently large, $\sqrt{n-1}-2\leq |\lambda_n| \leq \sqrt{n-1}+2$.
\end{lemma}
\begin{proof}
 The upper bound on $|\lambda_n|$ follows immediately from Lemma \ref{lambda_1 bounds} and the well-known fact $\lambda_1\geq |\lambda_n|$ for any graph. Now to get the lower bound, since $G$ is the outerplanar graph on $n$ vertices that maximizes spread, we have
\begin{align*}
    S(G)&\geq S(K_{1,n-1})\\
    &=\lambda_1(K_{1,n-1})-\lambda_n(K_{1,n-1})\\
    &=\sqrt{n-1}-(-\sqrt{n-1})\\
    &=2\sqrt{n-1}.
\end{align*}
So
\[
2\sqrt{n-1} \leq \lambda_1(G) - \lambda_n(G) \leq \sqrt{n}+1 - \lambda_n(G) < \sqrt{n-1}+2 - \lambda_n(G).
\]
Hence we have 
\[
-\lambda_n \geq \sqrt{n-1}-2.
\]
\end{proof}
Essentially the same proof also gives a lower bound for $\lambda_1$.
\begin{corollary}\label{lambda_1 lower bound}
For $n$ large enough we have $\lambda_1\geq \sqrt{n-1}-2$. 
\end{corollary}
We shall use Lemma \ref{lambda_n crude bound} to obtain a lower bound on the degree of each vertex.
\begin{lemma}\label{d_u lower bound z_u}
Let $u$ be an arbitrary vertex in $G$. Then $d_u>|\mathbf{z}_u|n-\mathcal{O}(\sqrt{n})$ and $d_u>\mathbf{x}_u n-\mathcal{O}(\sqrt{n})$.
\end{lemma}
\begin{proof}
We will show the first part explicitly. When $\lambda_n$ is the smallest eigenvalue for our graph, we have
\begin{align*}
    |\lambda_n^2\mathbf{z}_u|&=\left|\sum\limits_{y\sim u}\sum\limits_{v\sim y} \mathbf{z}_v\right|\\
    &\leq d_u+\sum\limits_{y\sim u}\sum_{\substack{v\sim y\\ v\not=u}}|\mathbf{z}_v|.
    \end{align*}
Recall that an outerplanar graph cannot have a $K_{2,3}$. This implies every vertex in $G$ has at most two neighbors in $N(u)$, meaning the eigenvector entry for each vertex contained in the neighborhood of $N(u)$ can be counted at most twice. Hence  $\sum\limits_{y\sim u}\sum\limits_{v\sim y}|\mathbf{z}_v|\leq 2\sum\limits_{v\not=w}|\mathbf{z}_v|$. Note \[|\lambda_n| |\mathbf{z}_v|\leq \sum\limits_{v'\sim v}|\mathbf{z}_{v'}|\leq \sum\limits_{v'\sim v} 1=d_v.\] So we have
 $$\sum_{y\sim u} \sum_{v\sim y } |\mathbf{z}_v| \leq 2\sum_{v \not=w} |\mathbf{z}_v|\leq  \frac{2}{|\lambda_n|}\sum\limits_{v\not=w}d_v  \leq \frac{4e(G)}{|\lambda_n|}\leq \frac{4(2n-3)}{|\lambda_n|},$$ as $e(G)\leq 2n-3$ by outerplanarity. Combining and using Lemma \ref{lambda_n crude bound}, we have 
 \[
 (\sqrt{n-1}-2)^2|\mathbf{z}_u |\leq |\lambda_n^2 \mathbf{z}_u| \leq d_u + \frac{4(2n-3)}{|\lambda_n|} \leq d_u + \frac{8n}{\sqrt{n-1}-2},
 \]
 
 for $n$ sufficiently large. Isolating $d_u$ gives the result.  A similar proof can be written to justify the lower bound with respect to $\mathbf{x}$ and we omit these details. 
 

\end{proof}

\begin{lemma}\label{lmac3}
We have $d_w>n-\mathcal{O}(\sqrt{n})$ and $d_{w'}>n-\mathcal{O}(\sqrt{n})$. For every other vertex $u$ we get $|\mathbf{z}_u|, \mathbf{x}_u = \mathcal{O}\left(\frac{1}{\sqrt{n}}\right)$, for n sufficiently large.
\end{lemma}
\begin{proof}
The bound on $d_w$ and $d_{w'}$ follows immediately from the previous lemma and the normalization that $|\mathbf{z}_{w'}| = \mathbf{x}_w=1$. Now consider any other vertex $u$. We know that $G$ contains no $K_{2,3}$ and hence can have at most $2$ common neighbors with $w$. Thus $d_u = \mathcal{O}(\sqrt{n})$. By Lemma \ref{d_u lower bound z_u}, we have that there are constants $c_1$ and $c_2$ such that 
\[
c_1\sqrt{n} > d_u\mathbf{x}_u - c_2\sqrt{n},
\]
and 
\[
c_1\sqrt{n} > d_u|\mathbf{z}_u| - c_2\sqrt{n},
\]
for sufficiently large $n$. This implies the result.
\end{proof}

If $w$ and $w'$ were distinct vertices, then for sufficiently large $n$ they would share many neighbors, contradicting outerplanarity. Hence we  immediately have the following important fact.
\begin{corollary}\label{w and w'}
For $n$ sufficiently large we have $w=w'$.
\end{corollary}
Hence from now on, we will denote the vertex in Corollary \ref{w and w'} by $w$, and furthermore for the remainder of the paper we will also assume that $\mathbf{z}_w = 1$ without loss of generality. Before deriving our next result quantifying the other entries of $\mathbf{z}$, we first need to define an important vertex set.
\begin{definition}\label{B def}
Recall that $w$ is the fixed vertex of maximum degree in $G$. Let $B=V(G)\backslash(N(w)\cup\{w\})$.
\end{definition}
Now we consider the $\mathbf{z}_u$ eigenvector entries of vertices in $B$. At the moment we know that each eigenvector entry for a vertex in $B$ has order at most $\frac{1}{\sqrt{n}}$. The next lemma shows that in fact the sum of all of the eigenvector entries of $B$ has this order. 

\begin{lemma}\label{lsum}For $n$ large enough, we have that 
$ \sum\limits_{u\in B}|\mathbf{z}_u|$ and $ \sum\limits_{u\in B}\mathbf{x}_u$ are each $ \mathcal{O}\left(\frac{1}{\sqrt{n}}\right)$.
\end{lemma}

\begin{proof}
Let $u\in B$. Since $u$ is not adjacent to $w$, all of the neighbors of $u$ have eigenvector entry of order at most $1/\sqrt{n}$ and the size of $B$ is also of order at most $1/\sqrt{n}$ by Lemma \ref{lmac3}. Hence there is a constant $C$ such that 
\[
|\lambda_n||\mathbf{z}_u| \leq \sum_{v\sim u}|\mathbf{z}_v| \leq \frac{C d_u}{\sqrt{n}},
\]
and $|B| \leq C\sqrt{n}$.  Now 
$$ \sum_{u\in B} |\mathbf{z}_u| \leq \frac{1}{|\lambda_n|}\sum_{u\in B}\left(\frac{Cd_u}{\sqrt{n}}\right)\leq \frac{C}{|\lambda_n| \sqrt{n}}(e(B, V(G)\setminus B)+2e(B)).$$
Each vertex in $B$ is connected to at most two vertices in $N(u)$, so $e(B, V(G)\setminus B)\leq 2|B|\leq 2C\sqrt{n}$. The graph induced on $B$ is outerplanar, so $e(B)\leq 2|B|-3<2C\sqrt{n}$. Finally, using Lemma \ref{lambda_n crude bound}, we get the required result. A slightly modified version of this argument proves the bound on $ \sum\limits_{u\in B}\mathbf{x}_u$.
\end{proof}

We will use Lemma \ref{lsum} to show that $B$ is empty, and this will complete the proof of Theorem \ref{vertex degree n-1}. First we define the following alteration of $G$. Let $t$ be an arbitrary vertex in $B$ (if it exists). 

\begin{definition}\label{G star}
Let $G^*$ be the graph defined such that its adjacency matrix $A^*$ satisfies $$A^*_{ij}=\begin{cases} 1 &$if $ i=w, j=t\\
0 &$if $ i=t, j\not=w\\
A_{ij} &$otherwise$
\end{cases}$$
That is, to get $G^*$ from $G$ we add an edge from $t$ to $w$ and remove all other edges incident with $t$. In particular, the only neighbor of the $t$ vertex in $G^*$ is $w$.
\end{definition}

\begin{lemma}\label{B empty}
For large enough $n$, $B$ is empty.
\end{lemma}

\begin{proof}
Assume for contradiction that $B$ is nonempty. Then there is a vertex $t$ such that $t\not\sim w$. Define $G^*$ as in Definition \ref{G star}. Furthermore, we define the vector $\mathbf{z}^*$ which slightly modifies $\mathbf{z}$ as follows.
\[
\mathbf{z}^*_u = \begin{cases}
\mathbf{z}_u & \mbox{if $u\not=t$}\\
-|\mathbf{z}_u| & \mbox{if $u=t$}
\end{cases}
\]
That is, if $\mathbf{z}_t<0$ then $\mathbf{z}^*$ is the same vector as $\mathbf{z}$ and otherwise we flip the sign of $\mathbf{z}_t$. Note that $(\mathbf{z}^*)^T \mathbf{z}^* = \mathbf{z}^T\mathbf{z}$. Now
\begin{align*}
S(G^*) - S(G) &\geq \left( \frac{\mathbf{x}^T A^*\mathbf{x}}{\mathbf{x}^T\mathbf{x}}- \frac{(\mathbf{z^*})^T A^*\mathbf{z^*}}{\mathbf{z}^T\mathbf{z}}\right) - \left(\frac{\mathbf{x}^T A\mathbf{x}}{\mathbf{x}^T\mathbf{x}}- \frac{\mathbf{z}^T A\mathbf{z}}{\mathbf{z}^T\mathbf{z}} \right)\\
&=\frac{2\mathbf{x}_t}{\mathbf{x}^T \mathbf{x}}\left(1-\sum\limits_{v\sim t}x_v\right)+\frac{2\mathbf{z}_t}{\mathbf{z}^T \mathbf{z}}\left(\mathrm{sgn}(\mathbf{z}_t)+\sum\limits_{v\sim t}\mathbf{z}_v\right) \\ 
&\geq \frac{2\mathbf{x}_t}{\mathbf{x}^T \mathbf{x}}\left(1-\sum\limits_{v\sim t}x_v\right)+\frac{2|\mathbf{z}_t|}{\mathbf{z}^T \mathbf{z}}\left(1-\left|\sum\limits_{v\sim t}\mathbf{z}_v\right|\right),
\end{align*}
where $\mathrm{sgn}(\mathbf{z}_t)$ equals $1$ if $\mathbf{z}_t > 0$ and $-1$ otherwise.
\[
\left|\sum\limits_{v\sim t}\mathbf{z}_v \right| \leq \sum_{v\sim t}|\mathbf{z}_v| \leq \sum_{\substack{v\sim t\\v\not\in B}}|\mathbf{z}_v| + \sum_{v\in B} |\mathbf{z}_v|.
\]
There are at most $2$ terms in the first sum, and so by Lemmas \ref{lmac3} and \ref{lsum}, we have 
\[
\left|\sum\limits_{v\sim t}\mathbf{z}_v \right| = \mathcal{O}\left(\frac{1}{\sqrt{n}}\right).
\]
Similarly, we have 
\[
\sum\limits_{v\sim t}\mathbf{x}_v  = \mathcal{O}\left(\frac{1}{\sqrt{n}}\right).
\]
This implies $1-\sum\limits_{v\sim t}\mathbf{x}_v > 0$ and  $1-\left|\sum\limits_{v\sim t}\mathbf{z}_v\right|>0$ for $n$ large enough, which implies that 
 \[
\frac{2\mathbf{x}_t}{\mathbf{x}^T \mathbf{x}}\left(1-\sum\limits_{v\sim t}x_v\right)+\frac{2|\mathbf{z}_t|}{\mathbf{z}^T \mathbf{z}}\left(1-\left|\sum\limits_{v\sim t}\mathbf{z}_v\right|\right) > 0
 \]
 for $n$ large enough. Hence $S(G^*) > S(G)$, contradicting the assumption that $G$ is spread-extremal.
\end{proof}

We have finally achieved our goal for this section. Theorem \ref{vertex degree n-1} follows immediately from the definition of $B$ and the fact that it is empty, as implied by Lemma \ref{B empty}.

\section{Determining Graph Structure}
By Theorem \ref{vertex degree n-1}, the vertex $w$ has degree $n-1$, or equivalently, $K_{1,n-1}$ is a subgraph of $G$. Since $G$ is $K_{2,3}$-free, the graph induced by the neighborhood of $w$ has maximum degree at most $2$. Furthermore, this subgraph cannot contain a cycle, otherwise $G$ would contain a wheel-graph and this is a $K_4$-minor. Any graph of maximum degree at most $2$ that does not contain a cycle is a disjoint union of paths. Therefore, we know that $G$ is given by a $K_1 \vee F$ where $F$ is a disjoint union of paths. Our next task is to study $F$. To this end, we denote the number of edges in $F$ by $m$. Our main theorem, Theorem \ref{thm: main} is proved if we can show that $m = \Omega(n)$. Before doing this, we need a more accurate estimate for the eigenvector entries. 

\begin{lemma}\label{z_u} For any $u\not=w$, we have $\mathbf{z}_u= \frac{1}{\lambda_n}+ \frac{d_u-1}{\lambda_n^2}+\Theta\left(\frac{1}{n^{3/2}}\right)$.
\end{lemma}
\begin{proof}
As we are only considering outerplanar graphs, we have $d_u\in \{1,2,3\}$ for all $u\neq w$. Hence
\begin{align*}
\lambda_n \mathbf{z}_u&=\sum\limits_{y \sim u}\mathbf{z}_y\\
&=\mathbf{z}_w+\sum_{\substack{y\sim u\\ y\not=w}} \mathbf{z}_y \\
&=1-\Theta\left(\frac{1}{\sqrt{n}}\right) 
\end{align*}
by Lemma \ref{lmac3} and our normalization. 

Note that as $\lambda_n\mathbf{z}_u=1+\mathbf{z}_{u_1}+\mathbf{z}_{u_2}>0$ and $\lambda_n<0$, we must have $\mathbf{z}_u<0$. 
Next we repeat the argument to improve our estimate,
\begin{align*}
    \lambda_n \mathbf{z}_u&=\mathbf{z}_w+\sum_{\substack{y\sim u\\ y\not=w}} \mathbf{z}_y\\
    &=1+(d_u-1)\left(\frac{1}{\lambda_n}+\Theta\left(\frac{1}{n}\right)\right).
    \shortintertext{Now we apply our bounds on $\lambda_n$ in Lemma \ref{lambda_n crude bound} to get}
    \mathbf{z}_u&= \frac{1}{\lambda_n}+ \frac{d_u-1}{\lambda_n^2}+\Theta\left(\frac{1}{n^{3/2}}\right).
\end{align*}
\end{proof}
We can use equivalent reasoning to obtain a very similar approximation for the $\mathbf{x}_u$ entries,
\begin{lemma}\label{x_u}For any $u\not=w$, we have
 $\mathbf{x}_u= \frac{1}{\lambda_1}+\frac{d_u-1}{\lambda_1^2}+\Theta(\frac{1}{n^{3/2}})$ 
\end{lemma}
\begin{proof}
By the same logic as in the proof of Lemma \ref{z_u}, we have 
\begin{align*}\lambda_1 \mathbf{x}_u&=\sum\limits_{w \sim u}\mathbf{x}_w\\
&=\mathbf{x}_w+\sum_{\substack{y\sim u\\ y\not=w}} \mathbf{x}_y\\
&=1+\Theta\left(\frac{1}{\sqrt{n}}\right) \text{by Corollary \ref{lmac3}}.
\end{align*}
So $\mathbf{x}_u=\frac{1}{\lambda_1}+\Theta\left(\frac{1}{n}\right)$. Next we repeat the argument to improve our estimate,
\begin{align*}
    \lambda_1 \mathbf{x}_u&=\mathbf{x}_w+\sum_{\substack{y\sim u\\ y\not=w}} \mathbf{x}_y\\
    &=1+(d_u-1)\left(\frac{1}{\lambda_1}+\Theta \left(\frac{1}{n}\right)\right) .
\end{align*}
So $\mathbf{x}_u= \frac{1}{\lambda_1}+\frac{d_u-1}{\lambda_1^2}+\Theta\left(\frac{1}{n^{3/2}}\right)$ according to our bounds on $\lambda_1$ in Lemma \ref{lambda_1 bounds}.
\end{proof}
Using Lemma \ref{z_u}, we can now get tighter estimates on $\lambda_n$.

\begin{lemma}\label{refined lambda_n bound}
We have that $\lambda_n= -\sqrt{n-1}+\frac{m}{n-1} +  \Theta \left(\frac{m}{n^{3/2}}\right)$.
\end{lemma}
\begin{proof}
We first define the vector \[\mathbf{y_2}=\begin{bmatrix}1\\ -\frac{1}{\sqrt{n-1}}\\ -\frac{1}{\sqrt{n-1}}\\ \vdots\\ -\frac{1}{\sqrt{n-1}}\end{bmatrix}\] where $\mathbf{y_2}$ has $n$ entries and $w$ corresponds to the 1 entry. As $\lambda_n$ is the minimum Rayleigh quotient  we have 
\begin{align*}
    \lambda_n &\leq \frac{\mathbf{y_2}^TA\mathbf{y_2}}{\mathbf{y_2}^T\mathbf{y_2}}\\
    &=\frac{2\sum\limits_{i \sim j} (\mathbf{y}_{2})_i(\mathbf{y}_{2})_j}{2}\\
    &=(n-1)\left(-\frac{1}{\sqrt{n-1}}\right)+m\left(\frac{1}{n-1}\right)\\
    &=-\sqrt{n-1}+\frac{m}{n-1}.
\end{align*}
In order to show the lower bound on $\lambda_n$, we need to realize that $\lambda_n$ is the minimum possible Rayleigh quotient. So 
 \[\lambda_n= \frac{\mathbf{z}^TA\mathbf{z}}{\mathbf{z}^T\mathbf{z}}
    =\frac{2\sum\limits_{i\sim j}\mathbf{z}_i\mathbf{z}_j}{\mathbf{z}^T\mathbf{z}}
    = \frac{2\sum\limits_{w\sim k}\mathbf{z}_w\mathbf{z}_k}{\mathbf{z}^T\mathbf{z}}+\frac{2\sum\limits_{{i\sim j \, ,\, i,j\neq w}}\mathbf{z}_i\mathbf{z}_j}{\mathbf{z}^T\mathbf{z}}.\]
    
    Note the first term is the Rayleigh quotient for the star subgraph centered at the vertex $w$ of maximum degree. Hence it is bounded from below by $-\sqrt{n-1}$. More specifically, we have
    \[
    \frac{2\sum\limits_{w\sim k}\mathbf{z}_w\mathbf{z}_k}{\mathbf{z}^T\mathbf{z}} = \frac{\mathbf{z}^T A(K_{1,n-1})\mathbf{z}}{\mathbf{z}^T\mathbf{z}} \geq \lambda_n(A(K_{1,n-1})) = -\sqrt{n-1}.
    \]

    Applying Lemma \ref{z_u} we have
    
    \begin{align*}
   \lambda_n &\geq -\sqrt{n-1}+\frac{2\sum\limits_{\substack{i\sim j\\ i,j\not=w}}\left(\frac{1}{\lambda_n}+ \frac{d_i-1}{\lambda_n^2}+\Theta\left(\frac{1}{n^{3/2}}\right)\right)\left(\frac{1}{\lambda_n}+ \frac{d_j-1}{\lambda_n^2}+\Theta\left(\frac{1}{n^{3/2}}\right)\right)}{1+\left(\sum_{\ell\not= w} \frac{1}{\lambda_n} + \frac{d_\ell-1}{\lambda_n^2}+ \Theta\left(\frac{1}{n^{3/2}}\right)\right)^2}
   \end{align*}
    For $\ell \not=w$ we have $0\leq d_\ell-1 \leq 2$. Note that for $n$ large enough we have that $\frac{1}{\lambda_n} + \frac{d_u-1}{\lambda_n^2} + \Theta\left(\frac{1}{n^{3/2}}\right) < 0$, and so we have a lower bound if we replace all $d_\ell-1$ terms in the numerator by $2$ and in the denominator by $0$, and so we have 
    \begin{align*}
   \lambda_n &\geq -\sqrt{n-1}+\frac{2\sum\limits_{\substack{i\sim j\\ i,j\not=w}}\left(\frac{1}{\lambda_n}+ \frac{2}{\lambda_n^2}+\Theta\left(\frac{1}{n^{3/2}}\right)\right)\left(\frac{1}{\lambda_n}+ \frac{2}{\lambda_n^2}+\Theta\left(\frac{1}{n^{3/2}}\right)\right)}{1+\sum\limits_{\ell\not= w} \left(\frac{1}{\lambda_n} + \Theta\left(\frac{1}{n^{3/2}}\right)\right)^2}
   \\&=
   -\sqrt{n-1} + \frac{2m\left(\frac{1}{\lambda_n^2} + \Theta\left(\frac{1}{n^{3/2}}\right)\right)}{1+(n-1)\left(\frac{1}{\lambda_n^2} + \Theta\left(\frac{1}{n^{3/2}}\right)\right)}.
   \end{align*}
   Since $-\sqrt{n-1}-2\leq \lambda_n \leq \sqrt{n-1}+2$, we have that $\frac{1}{\lambda_n^2} = \frac{1}{n-1} + \Theta\left(\frac{1}{n^{3/2}}\right)$. Hence we have
   
   \begin{align*}
     \lambda_n &\geq -\sqrt{n-1} + \frac{2m\left(\frac{1}{n-1} + \Theta\left(\frac{1}{n^{3/2}}\right)\right)}{1+(n-1)\left(\frac{1}{n-1} +\Theta\left(\frac{1}{n^{3/2}}\right)\right)}\\
     &= -\sqrt{n-1}+ \frac{\frac{2m}{n-1} + \Theta\left(\frac{m}{n^{3/2}}\right)}{2 + \Theta\left(\frac{1}{n^{1/2}}\right)} \\
     &= -\sqrt{n-1}+ \left(\frac{m}{n-1} + \Theta\left(\frac{m}{n^{3/2}}\right)\right)\left(1 + \Theta\left(\frac{1}{\sqrt{n}}\right)\right),
   \end{align*}
   completing the proof.
 
\end{proof}

We claim a very similar result for $\lambda_1$.

\begin{lemma}\label{refined lambda_1 bound}
$\lambda_1=\sqrt{n-1}+\frac{m}{n-1}+\Theta\left(\frac{m}{n^{3/2}}\right)$.
\end{lemma}
\begin{proof}
First we prove the lower bound. Define the vector \[\mathbf{y_3}=\begin{bmatrix}1\\ \frac{1}{\sqrt{n-1}}\\ \frac{1}{\sqrt{n-1}}\\ \vdots\\ \frac{1}{\sqrt{n-1}}\end{bmatrix},\]
where $w$ corresponds to the entry 1. As $\lambda_1$ is the maximum possible Rayleigh quotient we have 
\begin{align*}
    \lambda_1 &\geq \frac{\mathbf{y}^TA\mathbf{y_3}}{\mathbf{y_3}^T\mathbf{y_3}}\\
    &=\frac{2\sum\limits_{i \sim j} (\mathbf{y}_3)_i(\mathbf{y}_3)_j}{2}\\
    &=(n-1)\left(\frac{1}{\sqrt{n-1}}\right)+m\left(\frac{1}{n-1}\right)\\
    &=\sqrt{n-1}+\frac{m}{n-1}.
\end{align*}
For the upper bound, we use the fact that $\lambda_1$ is the maximum possible Rayleigh quotient similarly to the previous lemma. So
\begin{align*}
    \lambda_1&= \frac{\mathbf{x}^TA\mathbf{x}}{\mathbf{x}^T\mathbf{x}}\\
    &=\frac{2\sum\limits_{i\sim j}\mathbf{x}_i\mathbf{x}_j}{\mathbf{x}^T\mathbf{x}}\\
    &= \frac{2\sum\limits_{w\sim k}\mathbf{x}_w\mathbf{x}_k}{\mathbf{x}^T\mathbf{x}}+\frac{2\sum{\substack{i\sim j \\ i,j\neq w}}\mathbf{x}_i\mathbf{x}_j}{\mathbf{x}^T\mathbf{x}}.
    \end{align*}
    As before, the first term is 
    \[
    \frac{\mathbf{x}^T A(K_{1,n-1})\mathbf{x}}{\mathbf{x}^T\mathbf{x}} \leq \lambda_1(A(K_{1,n-1})) = \sqrt{n-1}.
    \]
    
     Applying Lemma \ref{x_u} and using $0\leq d_u-1 \leq 2$ for all $u\not=w$ we have
    
    \begin{align*}
   \lambda_1 &\leq \sqrt{n-1}+\frac{2\sum_{\substack{i\sim j\\ i,j\not=w}}\left(\frac{1}{\lambda_1}+ \frac{d_i-1}{\lambda_1^2}+\Theta\left(\frac{1}{n^{3/2}}\right)\right)\left(\frac{1}{\lambda_1}+ \frac{d_j-1}{\lambda_1^2}+\Theta\left(\frac{1}{n^{3/2}}\right)\right)}{1+\left(\sum_{\ell\not= w} \frac{1}{\lambda_1} + \frac{d_\ell-1}{\lambda_1^2}+ \Theta\left(\frac{1}{n^{3/2}}\right)\right)^2}\\
   &\leq \sqrt{n-1}+\frac{2\sum_{\substack{i\sim j\\ i,j\not=w}}\left(\frac{1}{\lambda_1}+ \frac{2}{\lambda_1^2}+\Theta\left(\frac{1}{n^{3/2}}\right)\right)\left(\frac{1}{\lambda_1}+ \frac{2}{\lambda_1^2}+\Theta\left(\frac{1}{n^{3/2}}\right)\right)}{1+\sum_{\ell\not= w} \left(\frac{1}{\lambda_1} + \Theta\left(\frac{1}{n^{3/2}}\right)\right)^2}
   \end{align*}
    
   Using $\sqrt{n-1}-2\leq \lambda_1 \leq \sqrt{n-1}+2$, we have
   
   \begin{align*}
     \lambda_1 &\leq \sqrt{n-1} + \frac{2m\left(\frac{1}{n-1} + \Theta\left(\frac{1}{n^{3/2}}\right)\right)}{1+(n-1)\left(\frac{1}{n-1} +\Theta\left(\frac{1}{n^{3/2}}\right)\right)}\\
     &= \sqrt{n-1}+ \frac{\frac{2m}{n-1} + \Theta\left(\frac{m}{n^{3/2}}\right)}{2 + \Theta\left(\frac{1}{n^{1/2}}\right)} \\
     &= \sqrt{n-1}+ \left(\frac{m}{n-1} + \Theta\left(\frac{m}{n^{3/2}}\right)\right)\left(1 + \Theta\left(\frac{1}{\sqrt{n}}\right)\right),
   \end{align*}
   completing the proof.

\end{proof}

We are now in a position to prove our main theorem.

\begin{proof}[Proof of Theorem \ref{thm: main}]
As before, let $G_1$ be the graph $K_1\vee P_{n-1}$ and let $A_1$ be its adjacency matrix with the dominating vertex corresponding to the first row and column. Since $G$ is spread-extremal we must have $S(G) \geq S(G_1)$. We lower bound $S(G_1)$ using the vectors
\[\mathbf{v}_1=\begin{bmatrix}-1 + \sqrt{n}\\ 1\\ 1\\ \vdots\\ 1\end{bmatrix} \quad \mbox{and} \quad \mathbf{v}_2=\begin{bmatrix}-1-\sqrt{n}\\ 1\\ 1\\ \vdots\\ 1\end{bmatrix}.\]
Using these vectors and the Rayleigh principle, we have that 
\begin{align*}
S(G_1) &\geq \frac{ \mathbf{v}_1^T A_1 \mathbf{v}_1}{\mathbf{v}_1^T\mathbf{v}_1} - \frac{ \mathbf{v}_2^T A_1 \mathbf{v}_2}{\mathbf{v}_2^T\mathbf{v}_2} \\
& = \frac{2(n-1)(\sqrt{n}-1) + (n-2)(1)}{2n-2\sqrt{n}} - \frac{2(n-1)(-\sqrt{n}-1) + (n-2)(1)}{2n+2\sqrt{n}}\\
& = \frac{2(n-1)(\sqrt{n}-1) + (n-1)(1)}{2n-2\sqrt{n}} -\frac{1}{2n-2\sqrt{n}}  \\& - \frac{2(n-1)(-\sqrt{n}-1) + (n-1)(1)}{2n+2\sqrt{n}}+ \frac{1}{2n+2\sqrt{n}}\\
&= \frac{n-1}{\sqrt{n}-1} + \frac{n-1}{\sqrt{n}+1} - \frac{\sqrt{n}}{n^2-n} \geq 2\sqrt{n} - \frac{1}{n}.
\end{align*}
Combining this with Lemmas \ref{refined lambda_n bound} and \ref{refined lambda_1 bound}, we have 
\[
2\sqrt{n} - \frac{1}{n} \leq \lambda_1(G) - \lambda_n(G) = \left(\sqrt{n-1}+\frac{m}{n-1}+\Theta\left(\frac{m}{n^{3/2}}\right)\right) - \left( -\sqrt{n-1}+\frac{m}{n-1}+\Theta\left(\frac{m}{n^{3/2}}\right)\right).
\]
Therefore there exists a constant $C$ such that for $n$ large enough we have 
\[
2\sqrt{n} - \frac{1}{n} \leq 2\sqrt{n-1} + \frac{C\cdot m}{n^{3/2}}.
\]
Since $\sqrt{n}- \sqrt{n-1} = \Omega(n^{-1/2})$, rearranging gives the final result.

\end{proof}

\section*{Acknowledgements}
The authors would like to thank the Community of Mathematicians and Statisticians Exploring Research (Co-MaStER) program at Villanova, through which this research was started.
\bibliographystyle{plain}
\bibliography{bib}

\begin{thebibliography}{10}

\bibitem{AP}
Tatjana~M Aleksi{\'c} and Miroslav Petrovi{\'c}.
\newblock Cacti whose spread is maximal.
\newblock {\em Graphs and Combinatorics}, 31(1):23--34, 2015.

\bibitem{ADLR}
Enide Andrade, Geir Dahl, Laura Leal, and Mar{\'\i}a Robbiano.
\newblock New bounds for the signless {L}aplacian spread.
\newblock {\em Linear Algebra and its Applications}, 566:98--120, 2019.

\bibitem{variable}
Mustapha Aouchiche, Francis~K Bell, Dragi{\v{s}}a Cvetkovi{\'c}, Pierre Hansen,
  Peter Rowlinson, Slobodan~K Simi{\'c}, and Dragan Stevanovi{\'c}.
\newblock Variable neighborhood search for extremal graphs. 16. some
  conjectures related to the largest eigenvalue of a graph.
\newblock {\em European Journal of Operational Research}, 191(3):661--676,
  2008.

\bibitem{B}
Iwo Biborski.
\newblock Note on the spread of real symmetric matrices with entries in fixed
  interval.
\newblock {\em Linear Algebra and its Applications}, 632:246--257, 2022.

\bibitem{BR}
Barry~N Boots and Gordon~F Royle.
\newblock A conjecture on the maximum value of the principal eigenvalue of a
  planar graph.
\newblock {\em Geographical analysis}, 23(3):276--282, 1991.

\bibitem{BRTU}
Jane Breen, Alex~WN Riasanovsky, Michael Tait, and John Urschel.
\newblock Maximum spread of graphs and bipartite graphs.
\newblock {\em arXiv preprint arXiv:2109.03129}, 2021.

\bibitem{CV}
Dasong Cao and Andrew Vince.
\newblock The spectral radius of a planar graph.
\newblock {\em Linear algebra and its applications}, 187:251--257, 1993.

\bibitem{CR}
Dragi{\v{s}}a Cvetkovi{\'c} and Peter Rowlinson.
\newblock The largest eigenvalue of a graph: A survey.
\newblock {\em Linear and multilinear algebra}, 28(1-2):3--33, 1990.

\bibitem{Deutsch}
Emeric Deutsch.
\newblock On the spread of matrices and polynomials.
\newblock {\em Linear Algebra and Its Applications}, 22:49--55, 1978.

\bibitem{DM}
Zden{\v{e}}k Dvo{\v{r}}{\'a}k and Bojan Mohar.
\newblock Spectral radius of finite and infinite planar graphs and of graphs of
  bounded genus.
\newblock {\em Journal of Combinatorial Theory, Series B}, 100(6):729--739,
  2010.

\bibitem{EZ}
Mark~N Ellingham and Xiaoya Zha.
\newblock The spectral radius of graphs on surfaces.
\newblock {\em Journal of Combinatorial Theory, Series B}, 78(1):45--56, 2000.

\bibitem{FF}
Yi-Zheng Fan and Shaun Fallat.
\newblock Edge bipartiteness and signless {L}aplacian spread of graphs.
\newblock {\em Applicable Analysis and Discrete Mathematics}, pages 31--45,
  2012.

\bibitem{FWG}
Yi-Zheng Fan, Yi~Wang, and Yu-Bin Gao.
\newblock Minimizing the least eigenvalues of unicyclic graphs with application
  to spectral spread.
\newblock {\em Linear Algebra and Its Applications}, 429(2-3):577--588, 2008.

\bibitem{FWL}
Yi-Zheng Fan, Jing Xu, Yi~Wang, and Dong Liang.
\newblock The {L}aplacian spread of a tree.
\newblock {\em Discrete Mathematics and Theoretical Computer Science},
  10(1):79--86, 2008.

\bibitem{GHK}
David~A Gregory, Daniel Hershkowitz, and Stephen~J Kirkland.
\newblock The spread of the spectrum of a graph.
\newblock {\em Linear Algebra and its Applications}, 332:23--35, 2001.

\bibitem{GZ}
Haiyan Guo and Bo~Zhou.
\newblock On adjacency-distance spectral radius and spread of graphs.
\newblock {\em Applied Mathematics and Computation}, 369:124819, 2020.

\bibitem{HS}
Yuan Hong and Jin-Long Shu.
\newblock Sharp lower bounds of the least eigenvalue of planar graphs.
\newblock {\em Linear Algebra and Its Applications}, 296(1-3):227--232, 1999.

\bibitem{JKW}
Charles~R Johnson, Ravinder Kumar, and Henry Wolkowicz.
\newblock Lower bounds for the spread of a matrix.
\newblock {\em Linear Algebra and Its Applications}, 71:161--173, 1985.

\bibitem{LZZ}
Xueliang Li, Jianbin Zhang, and Bo~Zhou.
\newblock The spread of unicyclic graphs with given size of maximum matchings.
\newblock {\em Journal of Mathematical Chemistry}, 42(4):775--788, 2007.

\bibitem{LN}
Huiqiu Lin and Bo~Ning.
\newblock A complete solution to the {C}vetkovi{\'c}--{R}owlinson conjecture.
\newblock {\em Journal of Graph Theory}, 97(3):441--450, 2021.

\bibitem{LMG}
Zhen Lin, Lianying Miao, and Shu-Guang Guo.
\newblock The ${A}_\alpha$-spread of a graph.
\newblock {\em Linear Algebra and its Applications}, 606:1--22, 2020.

\bibitem{LM}
Bolian Liu and Liu Mu-Huo.
\newblock On the spread of the spectrum of a graph.
\newblock {\em Discrete mathematics}, 309(9):2727--2732, 2009.

\bibitem{LL}
Muhuo Liu and Bolian Liu.
\newblock The signless {L}aplacian spread.
\newblock {\em Linear algebra and its applications}, 432(2-3):505--514, 2010.

\bibitem{Mirsky}
Leon Mirsky.
\newblock The spread of a matrix.
\newblock {\em Mathematika}, 3(2):127--130, 1956.

\bibitem{NylenTam}
Peter Nylen and Tin-Yau Tam.
\newblock On the spread of a {H}ermitian matrix and a conjecture of {T}hompson.
\newblock {\em Linear and Multilinear Algebra}, 37(1-3):3--11, 1994.

\bibitem{OLNK}
Carla~Silva Oliveira, Leonardo~Silva De~Lima, Nair Maria~Maia de~Abreu, and
  Steve Kirkland.
\newblock Bounds on the ${Q}$-spread of a graph.
\newblock {\em Linear algebra and its applications}, 432(9):2342--2351, 2010.

\bibitem{PBA}
Miroslav Petrovi{\'c}, Bojana Borovi{\'c}anin, and Tatjana Aleksi{\'c}.
\newblock Bicyclic graphs for which the least eigenvalue is minimum.
\newblock {\em Linear algebra and its applications}, 430(4):1328--1335, 2009.

\bibitem{SW}
AJ~RF and RJ~Schwenk.
\newblock On the eigenvalues of a graph.
\newblock {\em Selected Topics in Graph Theory, Academic Press, San Diego},
  pages 307--336, 1978.

\bibitem{alex}
Alex W~Neal Riasanovsky.
\newblock {\em Two Problems in Extremal Combinatorics}.
\newblock PhD thesis, Iowa State University, 2021.

\bibitem{Rowlinson}
Peter Rowlinson.
\newblock On the index of certain outerplanar graphs.
\newblock {\em Ars Combinatoria}, 29:221--225, 1990.

\bibitem{stanic}
Zoran Stani{\'c}.
\newblock {\em Inequalities for graph eigenvalues}, volume 423.
\newblock Cambridge University Press, 2015.

\bibitem{stevanovic}
Dragan Stevanovic.
\newblock {\em Spectral radius of graphs}.
\newblock Academic Press, 2014.

\bibitem{TT}
Michael Tait and Josh Tobin.
\newblock Three conjectures in extremal spectral graph theory.
\newblock {\em Journal of Combinatorial Theory, Series B}, 126:137--161, 2017.

\bibitem{Thompson}
Robert~C Thompson.
\newblock The eigenvalue spreads of a {H}ermitian matrix and its principal
  submatrices.
\newblock {\em Linear and Multilinear Algebra}, 32(3-4):327--333, 1992.

\bibitem{john}
John~C Urschel.
\newblock {\em Graphs, Principal Minors, and Eigenvalue Problems}.
\newblock PhD thesis, MASSACHUSETTS INSTITUTE OF TECHNOLOGY, 2021.

\bibitem{WZS}
Bing Wang, Ming-qing Zhai, and Jin-long Shu.
\newblock On the spectral spread of bicyclic graphs with given girth.
\newblock {\em Acta Mathematicae Applicatae Sinica, English Series},
  29(3):517--528, 2013.

\bibitem{WZL}
Junliang Wu, Pingping Zhang, and Wenshi Liao.
\newblock Upper bounds for the spread of a matrix.
\newblock {\em Linear algebra and its applications}, 437(11):2813--2822, 2012.

\bibitem{WS}
Yarong Wu and Jinlong Shu.
\newblock The spread of the unicyclic graphs.
\newblock {\em European Journal of Combinatorics}, 31(1):411--418, 2010.

\bibitem{YRY}
Lihua You, Liyong Ren, and Guanglong Yu.
\newblock Distance and distance signless {L}aplacian spread of connected
  graphs.
\newblock {\em Discrete Applied Mathematics}, 223:140--147, 2017.

\bibitem{YL}
Zhifu You and Bolian Liu.
\newblock The {L}aplacian spread of graphs.
\newblock {\em Czechoslovak mathematical journal}, 62(1):155--168, 2012.

\bibitem{YZLWS}
Guanglong Yu, Hailiang Zhang, Huiqiu Lin, Yarong Wu, and Jinlong Shu.
\newblock Distance spectral spread of a graph.
\newblock {\em Discrete Applied Mathematics}, 160(16-17):2474--2478, 2012.

\bibitem{yuan1}
Hong Yuan.
\newblock On the spectral radius and the genus of graphs.
\newblock {\em Journal of Combinatorial Theory, Series B}, 65(2):262--268,
  1995.

\bibitem{yuan2}
Hong Yuan.
\newblock Upper bounds of the spectral radius of graphs in terms of genus.
\newblock {\em Journal of Combinatorial Theory, Series B}, 74(2):153--159,
  1998.

\end{thebibliography}

\end{document}